\documentclass[letterpaper, 10 pt, conference]{ieeeconf}  

\IEEEoverridecommandlockouts                              
\usepackage{latexsym,amssymb,amsmath, amsbsy,amsopn, amstext}
\usepackage{graphicx,epsfig,tikz,pgf,hyperref,cancel,color,float,ifpdf}
\usepackage{extarrows}
\let\labelindent\relax
 \usepackage{enumitem}
\newcommand{\ud}{\mathrm{d}}
\def\e{\epsilon}
\def\l{\lambda}
\def\d{\delta}

\def\a{\alpha}
\def\b{\beta}
\def\m{\mu}

\def\S{\Sigma}
\def\b{\beta}

\def\R{\mathbb{R}}
\def\Z{\mathbb{Z}}

\def\ra{\rightarrow}
\def\lra{\longrightarrow}
\def\beq{\begin{equation}}
\def\eeq{\end{equation}}

\newtheorem{dfn}{\bfseries Definition}
\newtheorem{prp}{\bfseries Proposition}
\newtheorem{ass}{\it Assumption}
\newtheorem{thm}{\bfseries Theorem}

\newtheorem{cor}{\bfseries Corollary}

\newtheorem{rem}{\bfseries Remark}

\title{\bf
Passivity Degradation In Discrete Control Implementations: An Approximate Bisimulation Approach
}

\author{Xiangru Xu$^{1}$, Necmiye Ozay$^{1}$, Vijay Gupta$^{2}$
\thanks{$^{1}$Department of Electrical Engineering and Computer Science, University of Michigan, Ann Arbor, MI, 48109. {\tt\small \{xuxiangr, necmiye\}@umich.edu}, $^{2}$Department of Electrical Engineering, University of Notre Dame, Notre Dame, IN, 46556.
{\tt\small vgupta2@nd.edu}}
\thanks{The work of X. Xu was supported by NSF grant CNS-1239037. The work of N. Ozay was supported in part by NSF grant CNS-1446298.  This is an extended
version of the paper \cite{xupasscdc15} to appear in the Proceedings of the 54th IEEE Conference on Decision and Control (CDC), Osaka, Japan, December 15-18, 2015.}
}

\begin{document}

\maketitle
\thispagestyle{empty}
\pagestyle{empty}

\begin{abstract}

\color{black}{In this paper, we present some preliminary results for compositional analysis of heterogeneous systems containing both discrete state  models and continuous systems using consistent notions of dissipativity and passivity. We study the following problem:
given a physical plant model and a continuous feedback controller designed using traditional control techniques, how is the closed-loop passivity affected when the continuous controller is replaced by a discrete (i.e., symbolic) implementation within this framework? Specifically, we give quantitative results on performance degradation when the discrete control implementation is approximately bisimilar
to the continuous controller, and based on them, we provide conditions that guarantee the boundedness property of the closed-loop system.}


\end{abstract}

\section{Introduction}\label{sec:introduction}


Consider a cyber-physical system in which certain elements (e.g. the plant to be controlled) are physical, while other elements (e.g. controllers) are implemented in software. How do we systematically analyze and design such heterogeneous systems? In spite of the many important steps that have been taken to answer this question, more work remains to be done. We present an approach that seeks to extend the powerful tools of passivity based analysis and design to such heterogeneous systems. Specifically, the question we consider is - suppose a continuous controller has been designed to ensure specified passivity indices for the closed-loop system, {\color{black}what guarantees on passivity can be preserved if the controller is implemented in software?}
Our answer quantitatively shows how much passivity can be inherited from the original design that was done in the continuous domain when the symbolic control implementation is approximately bisimilar to the continuous controller.

Passivity, and its generalization dissipativity, are traditional tools in control theory. They have seen a recent resurgence for design of large scale systems since they offer the crucial property of compositionality - two passive systems in feedback configuration remain passive \cite{Khalil2000}. This property is not satisfied, e.g., by stability. We refer the reader to texts such as~\cite{Khalil2000,BrogDissBook07,JCWillems72} for a background in passivity. We note here that, perhaps inspired by cyber-physical systems, many important recent results have considered effects such as quantization~\cite{ZhuQuanSwitchHSCC12}, delays~\cite{chopra08,k08,matiakis06,Yu2013}
and packet drops~\cite{wang14} induced by communication networks and discretization of time \cite{LailaECC02,stramigioli2005sampled,OishiCDC10} on passivity. Each of these effects require an extension of the original definition of passivity and/or extra assumptions to be imposed on the set-up; for instance, \cite{OishiCDC10} requires the gain from the input to the derivative of the output to be bounded in a suitable sense. Under similar assumptions, we show that it is possible to reason about passivity of discrete-state systems and their interconnections with continuous plants. Our work is also related (yet, complementary) to the line of work on passivity of switched and hybrid systems~\cite{bbb08,pjs98,zbs01,zh08}, the input/output gain notions for finite-state systems \cite{tarraf2014input}, and analysis of control implementations \cite{Ny2010robustness,Anta10}.

Recently proposed symbolic models for dynamical systems provide a unified framework for studying the interactions of software and physical phenomena (see for instance, \cite{PauloBook09, PolaAuto08}, and the references therein). Such symbolic models are also used in control software synthesis from high-level specifications \cite{pola2012integrated, ZamaniTAC12, liu2014abstraction}. The basic workflow in these approaches is (i) to compute an approximate symbolic model of the plant based on (bi)simulation relations, (ii) to synthesize a discrete controller for the symbolic plant model, (iii) to refine the discrete controller and compose it with the original plant in feedback. By construction these discrete controllers can be implemented in software and guarantee that the closed-loop system satisfies the desired high-level {\color{black} specifications}. However, a controller designed using these approaches essentially includes an internal model of the plant it is designed for~\cite{PauloBook09}. Therefore, the complexity of the controller grows (exponentially) with the dimension of the state-space of the plant, which remains a major limitation for large scale applicability of these techniques for complex systems.

For specifications like stability or passivity, it is often possible to achieve desired performance with  classical controllers such as PID, lead-lag compensators, or controllers with simple state space representations. If these controllers are then implemented as software modules, it is of interest to ask what is the resulting effect on such specifications. This paper addresses that question. The basic workflow that we consider is (i) to synthesize a continuous controller for the continuous plant model to satisfy specified passivity indices in closed loop, (ii) to compute an approximate symbolic model of the controller based on bisimulation relations, and (iii) to compose the discrete controller with the original plant in feedback. Our results provide a relation between passivity indices of original continuous feedback loop, the bisimulation parameters, and the passivity indices of the new setup with the discrete controller and continuous plant. These results can then be used to guide the software implementation of the controller so that desired performance (in terms of passivity) can be guaranteed in the new setup. We note that the controllers we consider are arbitrary (apart from constraints such as stability) and are designed in continuous space (where the set of available tools is much richer). 


\emph{Notation:}
$\Z,\Z_+,\Z_0,\R,\R^+,\R_0^+$ denote the set of integer, positive integer, non-negative integer, real, positive real and non-negative real numbers, respectively. $\R^n$ denotes the space of $n$-dimensional real vectors.
$\|x\|$ and $\|x\|_2$ denote the $\ell_\infty$ and $\ell_2$ norm of a vector $x\in \R^n$, respectively. {\color{black} For continuous (discrete) time signal $\bf{x}:\R_0^+\rightarrow\R^n$ ($\bf{x}:\Z_0\rightarrow\R^n$), $x(t)$ ($x[k]$) denotes its value at time $t$ (time step $k$).}
For any $A\subseteq \R^n$ and $\mu>0$, $[A]_\m:=\{a\in A|a_i=k_i\m,k_i\in \Z,i=1,2,...,n\}$.
A relation $R\subset A\times B$ is identified with the map $R:A\ra 2^B$, which is defined by $b\in R(a)$ if and only if $(a,b)\in R$. For a set $S$, the set $R(S)$ is defined as $R(S)=\{b\in B: \exists a\in S, (a,b)\in R\}$. Given a relation $R\subset A\times B$, $R^{-1}$ denotes the inverse relation of $R$, i.e., $R^{-1}:=\{(b,a)\in B\times A: (a,b)\in R\}$.

\section{Preliminaries}\label{sec:preliminary}


\subsection{Control Systems and Transition Systems}

A continuous-time control system is a tuple $\S=(X,U,Y,f,h)$ where $X\subseteq \R^n$ is a set of states, $U\subseteq\R^m$ is a set of inputs, $Y\subseteq \R^m$ is a set of outputs, $f:X\times U\ra \R^n$  and $h:X \times U\ra \R^m$ are both continuous maps. The state, input and output of $\S$ at time $t\in\R_0^+$ is denoted by $x(t)$, $u(t)$, $y(t)$, respectively, and their evolution is governed by:
\begin{equation}\label{ctd}
\begin{aligned}
\dot{x}(t) &=f(x(t),u(t)),\\
y(t) & =h(x(t),u(t)), \;\; \forall t \in \R_0^+.
\end{aligned}
\end{equation}
We assume that
$f$ satisfies the standard conditions such that, given any sufficiently regular control input signal ${\bf u}:[0,T]\rightarrow U$ with $T\ge 0$ and any initial condition $x_0\in X$, there exists a unique state (and output) trajectory $\bf{x}$ (and $\bf{y}$) defined on $[0,T]$ satisfying $x(0)=x_0$ and Eq.~\eqref{ctd}.
Denote by $\textbf{x}(\tau,x_0,\textbf{u})$  the state reached at time $\tau$ under the input $\textbf{u}$ from the initial state $x_0$ of $\S$.
The continuous-time system $\S$ is called \emph{incrementally input-to-state stable} ($\d$-ISS) ~\cite{angeli2002lyapunov} if it is forward complete and there exist functions $\b_1\in \mathcal{KL}$ and $\b_2\in \mathcal{K}_{\infty}$~\cite{Khalil2000} such that for any $t\in\R_0^+$, any initial state $x_1,x_2\in \R^n$ and any input $\textbf{u},\textbf{v}$, the following condition is satisfied:
\begin{equation}\label{eqniss}
\|\textbf{x}(t,x_1,\textbf{u})-\textbf{x}(t,x_2,\textbf{v})\|\leq \b_1(\|x_1-x_2\|,t)+\b_2(\|\textbf{u}-\textbf{v}\|_{\infty}).
\end{equation}

A discrete-time control system is a tuple $\S_d=(X,U,Y,f_d,h_d)$ where
$X\subseteq \R^n$ is a set of states, $U\subseteq\R^m$ is a set of inputs, $Y\subseteq \R^m$ is a set of outputs, $f_d:X\times U\ra X$ and $h_d:X \times U\ra \R^m$ are both continuous maps.
The state, input and output of $\S_d$ at time step $k\in \Z_0$ is denoted by $x[k]$, $u[k]$, $y[k]$, respectively, and their evolution is governed by:
\begin{equation}\label{dtd}
\begin{aligned}
x[k+1] &=f_d(x[k],u[k]),\\
y[k] & =h_d(x[k],u[k]), \;\; \forall k \in \Z_0.
\end{aligned}
\end{equation}
The state and output trajectories of the system $\S$ are discrete-time signals satisfying Eq.~\eqref{dtd}.




\begin{dfn}
A transition system is a quintuple $T=(Q,L,O,\lra,H)$, where:
\begin{itemize}
\item $Q$ is a set of states;
\item $L$ is a set of labels;
\item $O$ is a set of outputs;
\item $\lra\subset Q\times L\times Q$ is the transition relation;
\item $H:Q\times L\lra O$ is the output function.
\end{itemize}
\end{dfn}

A transition system $T$ is said to be metric if $Q,L,O$ are equipped with metrics $\textbf{d}_s: Q\times Q\ra \R_0^+$, $\textbf{d}_l: L\times L\ra \R_0^+$ and $\textbf{d}_o: O\times O\ra \R_0^+$, respectively. Denote an element $(q,\ell,p)\in\lra$ by $q\xlongrightarrow{\ell}p$.

Bisimulation is a binary relation between two transition systems, which, roughly speaking, requires the two systems match each other's behavior \cite{PauloBook09}.
Girard et al. generalized the conventional bisimulation to  $\e$-approximate bisimulation, which allows the states of two transition systems to be within certain bounds \cite{GirardTAC07}. Furthermore, to capture the input and output behaviors of transition systems, we  consider the following $(\e,\mu)$-approximate bisimulation relation adopted from \cite{julius2009approximate}.


\begin{dfn}\label{dfnemubi}
Given two metric transition systems $T_1=(Q_1,L,O,\xlongrightarrow[1]{},H_1)$ and $T_2=(Q_2,L,O,\xlongrightarrow[2]{},H_2)$  where the state sets $Q_1,Q_2$ are equipped with the same metric $\textbf{d}_s$ and the input set $L$ is equipped with the metric $\textbf{d}_l$, for any $\e,\mu\in\R^+$, a relation $R\subset Q_1\times Q_2$ is said to be an $(\e,\mu)$-approximate bisimulation relation between $T_1$ and $T_2$, if for any $(q_1,q_2)\in R$:
\begin{enumerate}
\item[(i)] $\textbf{d}_s(q_1,q_2)\leq \e$;
\item[(ii)] $q_1\xlongrightarrow[1]{\ell_1}p_1$ implies the existence of $\ell_2\in L$ such that $\textbf{d}_l(\ell_1,\ell_2)\leq \mu$, $q_2\xlongrightarrow[2]{\ell_2}p_2$ and $(p_1,p_2)\in R$;
\item[(iii)] $q_2\xlongrightarrow[2]{\ell_2}p_2$ implies the existence of
   $\ell_1\in L$ such that $\textbf{d}_l(\ell_1,\ell_2)\leq \mu$, $q_1\xlongrightarrow[1]{\ell_1}p_1$ and $(p_1,p_2)\in R$.
\end{enumerate}
\end{dfn}

If there exists an $(\e,\mu)$-approximate bisimulation relation $R$ between $T_1$ and $T_2$ such that $R(Q_1)=Q_2$ and $R^{-1}(Q_2)=Q_1$, $T_1$ is called $(\e,\mu)$-bisimilar to $T_2$, which is denoted as $T_1\cong^{(\e,\mu)}T_2$.






Given a continuous-time control system $\Sigma=(X,U,Y,f,h)$ with $X=\R^n,U=\R^m,Y=\R^m$ and a sampling time $\tau$, we define a transition system associated with the time-discretization of $\S$ as
%
$T_\tau(\Sigma)=(X_1,U_1,Y_1,\xlongrightarrow[1]{},H_1)$, which consists of:
\begin{itemize}
  \item $X_1=\R^n$;
  \item $U_1= \R^m$;
  \item $Y_1=\R^m$;
  \item $p\xlongrightarrow[1]{u} q$ if $\textbf{x}(\tau,p,\textbf{u})=q$ where $\textbf{u}:[0,\tau)\rightarrow \{u\}$, $u\in U_1$ (i.e., $\textbf{u}$ is a constant signal with value $u$);
  \item $H_1=h$.
\end{itemize}

We interpret the trajectories of $T_\tau(\Sigma)$ in discrete-time, that is, it has an equivalent representation in terms of a discrete-time control system as in \eqref{dtd}.


By further quantizing the state and input spaces of $\Sigma$, we obtain an infinitely countable transition system $T_{\tau\mu\eta}(\Sigma)=(X_2,U_2,Y_2,\xlongrightarrow[2]{},H_2)$ for some $\tau,\mu,\eta>0$:
\begin{itemize}
  \item $X_2=[\R^n]_\eta$;
  \item $U_2=[\R^m]_\mu$;
   \item $Y_2=[\R^m]_\mu$;
  \item $p\xlongrightarrow[1]{\hat{u}} q$ if $||\textbf{x}(\tau,p,\textbf{u})-q||\leq \eta/2$ where $\textbf{u}:[0,\tau)\rightarrow \{\hat{u}\}$, $\hat{u}\in U_2$;

  \item $H_2(\hat{x},\hat{u})=\hat y$ where $\|\hat y-h(\hat{x},\hat{u})\|\leq\mu/2$, $\hat y\in Y_2$.
\end{itemize}

The trajectories of $T_{\tau\mu\eta}(\Sigma)$ are also interpreted in discrete-time. The state, input and output of $T_{\tau\mu\eta}(\Sigma)$ at time step $k\in \Z_0$ is denoted by $\hat x[k]$, $\hat u[k]$, $\hat y[k]$, respectively and their (possibly non-deterministic) evolution is governed by:
\begin{equation}\label{symd}
\begin{aligned}
\hat x[k+1] & \in \{p\in X_2 \mid  \hat x[k] \xlongrightarrow[2]{\hat u[k]}p \},\\
\hat y[k] & \in H_2(\hat x[k],\hat u[k]), \;\; \forall k \in \Z_0.
\end{aligned}
\end{equation}

The following proposition is a direct application of Theorem 5.1 in \cite{PolaAuto08}.
\begin{prp}\label{thmsi}
	Consider a continuous-time control system $\S$ and any desired precision $\e>0$. If $\S$ is $\d$-ISS satisfying (\ref{eqniss}) and parameters $\tau,\eta,\mu>0$
	satisfy the following inequality
	\begin{equation}\label{eqnbisi}
	\beta_1(\e,\tau)+\beta_2(\mu)+\eta/2\leq \e,
	\end{equation}
	then $T_\tau(\S)\cong^{(\e,\mu)} T_{\tau\mu\eta}(\Sigma)$.
\end{prp}

We call the transition system $T_{\tau\mu\eta}(\S)$ with $T_\tau(\S)\cong^{(\e,\mu)} T_{\tau\mu\eta}(\S)$ a \emph{symbolic} model for $\S$. One nice property of this symbolic model is that its evolution can be chosen to be deterministic \cite{girard2013low}, which is appropriate for discrete software-based implementation.



%

\subsection{Dissipativity and Quasi-dissipativity}\label{secdis}

Some basic definitions about dissipativity and quasi-dissipativity are now given. 

\begin{dfn}\label{ddissCT}
A continuous-time control system $\S$ is called \emph{dissipative} with respect to a supply function $w(x,u,y)$ if there exists a continuous, positive semi-definite  storage function $V(x)$ such that the following (integral) dissipation inequality
\begin{equation}\label{con:dissipa}
V(x(t_2))-V(x(t_1))\leq \int_{t_1}^{t_2}w(x(s),u(s),y(s))\,\ud s
\end{equation}
is satisfied for all {\color{black}$t_1,t_2\in\R_0^+$} with $t_1<t_2$ and all inputs $u$.
\end{dfn}


A system is called \emph{input feed-forward output feedback passive} if it is dissipative with respect to the supply function $w(u,y)=u^Ty-\rho y^Ty-\nu u^Tu$ for some $\nu,\rho\in\R$ \cite{ZhuFeedbackACC14}.
Particularly, it is called \emph{passive} if $\nu=\rho=0$ and
\emph{very strictly passive} (VSP) if $\rho>0,\nu> 0$. The numbers $\nu,\rho$ are called \emph{passivity indices}, which reflect the excess or shortage of passivity of the system.
{\color{black}There are two points to emphasize. First, the indices $\rho$ and $\nu$ are not required to be non-negative in the input feed-forward output feedback passivity definition. Second, these indices are not unique for a given system; a system typically admits a family of indices.}


%


\begin{dfn}\label{ddissDT}
A discrete time control system $\S_d$ is called \emph{dissipative} with respect to the supply function $w(x,u,y)$ if there exists a continuous, positive semi-definite storage function $V(x)$ such that  the following dissipative inequality
\begin{equation}\label{dis:dissipa}
V(x[k+1])-V(x[k])\leq w(x[k],u[k],y[k])
\end{equation}
is satisfied for any $k\in\Z_0$ and any $u$.
\end{dfn}

Similar to the continuous case, we can define different dissipativity notions by choosing various forms of $w(x,u,y)$.


Unlike dissipative systems defined above, \emph{quasi-dissipative} (or almost-dissipative)  systems {\color{black}allow for} internal energy generation, and {\color{black}are defined as follows} \cite{DowSCL03,PolQuasiTAC04}.

\begin{dfn}
The continuous {\color{black}(resp. discrete)} time control system $\S$ {\color{black}(resp. $\S_d$)} is called quasi-dissipative with supply function $w(x,u,y)$ if there exists a constant $\alpha\geq 0$ such $\S$ {\color{black}(resp. $\S_d$)}  is dissipative with supply function  $w(x,u,y)+\alpha$.
\end{dfn}


It is clear that dissipative systems are quasi-dissipative with $\a=0$. If $w(x,u,y)=u^Ty-\nu u^Tu-\rho y^Ty$ with $\nu,\rho>0$, then the system is called
$(\nu,\rho,\alpha)$-\emph{very strictly quasi-passive} (VSQP). For example, a discrete-time VSQP system $\S_d$ satisfies the following dissipative inequality  for any $k\in\Z_0$:
\begin{align}
&V(x[k+1])-V(x[k])\nonumber\\
&\quad\leq u[k]^Ty[k]-\nu u[k]^Tu[k]-\rho y^T[k]y[k]+\a,\label{ieq:VSQP}
\end{align}
where $\nu,\rho,\a>0$ and $V(x)$ is {\color{black}positive semi-definite}. 


\begin{dfn}\label{dfndete}
System (\ref{dtd}) is called strongly finite-time detectable if there exist $N_0\in\Z_0,\kappa \in\R ^+$ such that for any $x[k],k\in\Z_0$ and for any input $u[i],k\leq i\leq k+N_0$,
\begin{equation}
\sum_{i=k}^{k+N_0}\|y[i]\|_2^2\geq \kappa\|x[k]\|_2^2.\label{strongdetec}
\end{equation}
\end{dfn}

Intuitively, strong finite-time detectability condition implies that large initial states result in large output signals. Next, we establish a connection between quasi-passivity and ultimate boundedness of discrete-time systems. This result will be later used in analysis of feedback interconnections. To the best of our knowledge, this result is new and can be of independent interest. A similar result for continuous-time systems is given in \cite{PolQuasiTAC04}.
\begin{thm}\label{thm:bound}
Given a strongly finite-time detectable discrete-time control system $\S_d$ that satisfies \eqref{ieq:VSQP} {\color{black}
with {\color{black}positive semi-definite}, radially unbounded $V$} and $\nu,\rho,\a>0$, if the input is bounded at all time and the initial state is also bounded, then the state of $\S_d$ is ultimately bounded, i.e., if $\sup_{\{i\geq k\}} \|u[i]\|_2\leq B_1$, $\|x[k]\|_2\leq B_2$ where $k$ is the initial time step and $B_1,B_2>0$, then there exists $D>0$ such that $\sup_{s\geq k}\|x[s]\|_2\leq D$.
\end{thm}
\begin{proof} Given in Appendix \ref{sec:prf1}.
\end{proof}

\begin{cor}\label{cor:bound}
Given a strongly finite-time detectable discrete-time control system $\S_d$ that satisfies $V(x[k+1])-V(x[k])
\leq -\rho y^T[k]y[k]+\a$ {\color{black}for any $k\in\Z_0$ with {\color{black}positive semi-definite}, radially unbounded $V(x)$ and $\rho,\a>0$}, if the input is bounded at all time and the initial state is also bounded, then the state of $\S_d$ is ultimately bounded.

\end{cor}



\subsection{Problem Setup}
In this subsection, we informally introduce the main problem considered in this paper.

Consider the setup in Figure~\ref{figSig2}, where a continuous-time control system $\S_1$, which corresponds to a physical plant, is connected in feedback with a discrete implementation $T_{\tau\m\eta}(\S_2)$, which is obtained from a designed continuous-time controller $\Sigma_{2}$ and satisfies $T_\tau(\S_2)\cong ^{(\e,\mu)}T_{\tau\m\eta}(\S_2)$ for some $\epsilon,\tau,\mu,\eta>0$.
Given the passivity indices of $\S_1$ and $\S_2$, our goal is to find bounds on the passivity indices of the closed-loop system $\S_c$. In this setup we assume that (i) the external reference inputs $w_i, i=1,2$ to the closed-loop system $\S_c$ are discrete-time {\color{black}signals (e.g.,  obtained via a digital sensor or user interface)}, and  (ii) all the discrete-time signals in the feedback loop are synchronized.  

\begin{figure}[!hbt]
\centering
  \includegraphics[width=.48\textwidth]{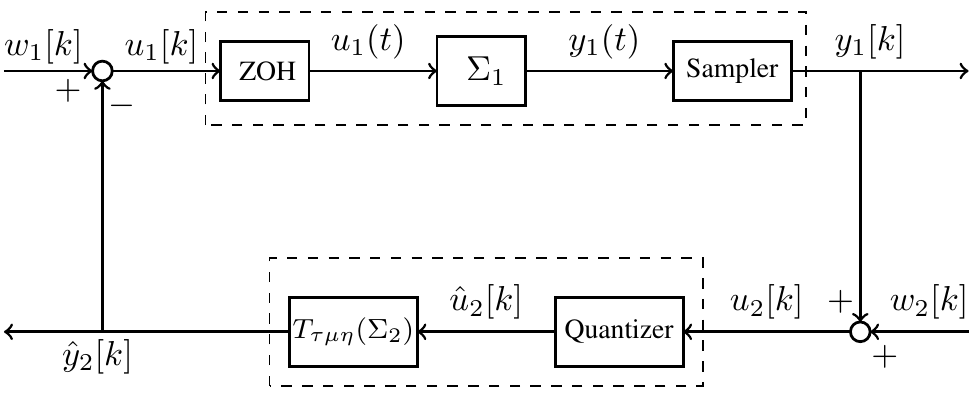}
\caption{The closed-loop system $\Sigma_c$ consists of $\S_1$ and $T_{\tau\mu\eta}(\S_2)$; ZOH is zero-order hold, sampler is the ideal sampler and quantizer is the uniform quantizer.  }\label{figSig2}
\end{figure}

The problem is solved in two steps. In Section \ref{sec:degradation}, we show that the passivity indices of $\Sigma_2$ {\color{black}induce certain} indices on $T_{\tau\m\eta}(\S_2)$. Then, in Section \ref{sec:feedback}, we show that the passivity properties of the closed-loop system with the controller $\Sigma_2$ are, in some sense, inherited by the system when $\Sigma_2$ is replaced by $T_{\tau\m\eta}(\S_2)$.

\section{{\color{black}Passivity Degradation of $\epsilon$-Approximate Bisimilar System}}\label{sec:degradation}

In this section, we study the relation of the passivity indices of a continuous-time control system $\Sigma$
to the passivity indices of its approximately bisimilar discrete version $T_{\tau\mu\eta}(\Sigma)$, which can be implemented in software. 

We suppose that the inputs to $\Sigma$ are piece-wise constant  signals in the set $\mathcal{U}_\tau:=\{{\bf u}:R_0^+\rightarrow U|u(t)=u((k-1)\tau),
\forall t\in [(k-1)\tau,k\tau),k\in \Z_+\}$ with $\tau$ the sampling time, and the following assumption from~\cite{OishiCDC10} holds:
\begin{ass}\label{assgain}
$\S$ satisfies the following gain inequality for any $t_1,t_2\in\R_0^+$ with $t_2>t_1$ and any admissible $u$:
\begin{equation}\label{eqngain}
\int_{t_1}^{t_2}\|\dot{y}(s)\|_2^2\,\ud s\leq\gamma^2 \int_{t_1}^{t_2}\|u(s)\|_2^2\,\ud s
\end{equation}
where $\gamma>0$ is a constant.
\end{ass}



The following result shows how passivity indices degrade under discretization and quantization.
\begin{thm}\label{thmpradis}
Consider a continuous time control system $\S=(X,U,Y,f,h)$ that is $\delta$-ISS and satisfies Assumption \ref{assgain}. For any $\e>0$,  suppose that parameters $\tau,\mu,\eta>0$ are chosen such that $T_\tau(\S)\cong ^{(\e,\mu)}T_{\tau\m\eta}(\S)$. Further, suppose that {\color{black}$\S$ is input feedforward output feedback passive with} passivity indices $(\nu,\rho)$
and storage function $V(x)$ that satisfies $|V(x_1)-V(x_2)|\leq L\|x_1-x_2\|^\theta$ for any $\|x_1-x_2\|\leq \epsilon$ and some constants $L,\theta>0$.
Then,
\begin{enumerate}[leftmargin=.4cm]
\item[(i)] $T_\tau(\S)$ satisfies the following (discrete time) passivity inequality for any $k\in\Z_0$:
\begin{align*}
& V(x[k+1])-V(x[k])\leq  \\
&\quad\quad  \tau u[k]^Ty[k]-\tau\nu'\|u[k]\|_2^2-\tau\rho'\|y[k]\|_2^2
\end{align*}
where
{\color{black}
\begin{align}\label{indextau1}
\begin{cases}
\nu'=&\nu-\gamma \tau(1+\lambda_1|\rho|)-\gamma^2\tau^2|\rho|,\\
\rho'=&\rho-\gamma\tau|\rho|/\lambda_1.
\end{cases}
\end{align}}
\item[(ii)] $T_{\tau\mu\eta}(\S)$ satisfies the  following (discrete time) passivity inequality  for any $k\in\Z_0$:
\begin{align*}
& V(\hat x[k+1])-V(\hat x[k])\leq\nonumber\\
& \quad\quad \tau\hat u[k]^T\hat y[k]-\tau\nu''\|\hat u[k]\|_2^2-\tau
\rho''\|\hat y[k]\|_2^2+\a
\end{align*}
where
{\color{black}
\begin{align}\label{indextau2}
\begin{cases}
\begin{array}{l}
\nu''=\nu-\tau\gamma-\frac{\lambda_2\sqrt{m}\mu}{4}
-|\rho|\tau\gamma(\tau\gamma+\sqrt{m}\mu\lambda_3+\lambda_4),\\
\rho''=\rho-|\rho|(\frac{\tau\gamma}{\lambda_4}+\sqrt{m}\mu\lambda_5),\\
\a=\tau\frac{\sqrt{m}\mu}{4\lambda_2}+\tau\mu|\rho|(\frac{\tau\gamma\sqrt{m}}{4\lambda_3}
+\frac{\sqrt{m}}{4\lambda_5}+\frac{m\mu}{4})+2L\epsilon^\theta.
\end{array}
\end{cases}
\end{align}
}
Here, $\l_1,\l_2,...,\l_5$ are arbitrary positive real numbers.
\end{enumerate}
\end{thm}
\begin{proof}
Given in Appendix  \ref{sec:prf2}.
\end{proof}

A few remarks are in order.
{\color{black}
\begin{rem} {\color{black} The numbers $\l_1,\l_2,...,\l_5$ provide some flexibility in choosing the passivity indices.}
Note that if $\lambda_1=1$, then $\nu'$ and $\rho'$ coincide with the degraded passivity indices $\epsilon_d,\delta_d$ given in \cite{OishiCDC10}, respectively; if $\mu=0,\lambda_4=\lambda_1$, then $\nu'',\rho''$ coincide with $\nu',\rho'$, respectively. This is intuitive since a smaller value of $\mu$ indicates that $U_2,Y_2$ are quantized with more precision, and $U_2,Y_2$ degenerate to $\R^m$ when $\mu=0$. However, the presence of $\alpha>0$ reflects the further passivity degradation under state quantization.
\end{rem}}
\begin{rem}
If $\textbf{x}(t,0,0)=0$, $\|x_0\|$ and $\|\textbf{u}\|_{\infty}$ are both bounded, then, by \eqref{eqniss} we can see that $\|\textbf{x}(t,x_0,\textbf{u})\|\leq \b_1(\|x_0\|,t)+\b_2(\|\textbf{u}\|_{\infty})$. Therefore, $\textbf{x}(t,x_0,\textbf{u})$ is bounded by some $M>0$ for any $t>0$, which implies that  the state $x[k]$ of $T_\tau(\S)$ is bounded by $M$ and $\hat x[k]$ of $T_{\tau\m\eta}(\S)$ is bounded by $M+\epsilon$.
\end{rem}
{\color{black}
\begin{rem}
Equations~\eqref{indextau1} and \eqref{indextau2} indicate that to ensure $\nu',\rho'>0$ or $\nu'',\rho''>0$, it is necessary that $\Sigma$ is VSP with $\nu,\rho>0$.
On the other hand, if $\nu,\rho>0$, it is not hard to find that to ensure $\nu',\rho'>0$ or $\nu'',\rho''>0$, $\tau$ should satisfy
$$
\tau<\frac{2\nu}{\gamma(\sqrt{8\nu\rho+1}+1)}:=\tau_{max}.
$$

Note that to ensure $T_\tau(\S)\cong ^{(\e,\mu)}T_{\tau\m\eta}(\S)$, $\tau$ should be chosen large enough such that inequality \eqref{eqnbisi} holds. Therefore, given precision $\e$, a necessary condition for $\nu',\rho'>0$ or $\nu'',\rho''>0$ is that there exist $\mu,\eta>0$ such that 
$\beta_1(\e,\tau_{max})+\beta_2(\mu)+\eta/2\leq \e$ holds.
	
\end{rem}}


\section{{\color{black}Passivity Analysis of Closed-loop System}}\label{sec:feedback}

In this section, we present the main result of the paper.


\begin{thm}\label{thm:feedback}
Consider the setup in Figure~\ref{figSig2} in which the system $\Sigma_{1}$ corresponds to the continuous time plant, and the system $T_{\tau\m\eta}(\S_2)$ (with $T_\tau(\S_2)\cong ^{(\e,\mu)}T_{\tau\m\eta}(\S_2)$ for some $\epsilon,\tau,\mu,\eta>0$) has been obtained from a system $\Sigma_{2}$, which corresponds to a continuous time controller. Suppose that $\S_i,i=1,2,$ satisfies Assumption \ref{assgain} with gain $\gamma_i>0$.
Furthermore, suppose that $\S_i,i=1,2,$ {\color{black}is input feedforward output feedback passive} with passivity indices $(\nu_i,\rho_i)$  and positive semi-definite storage function $V_i$, which satisfies $|V_i(x_1)-V_i(x_2)|\leq L\|x_1-x_2\|^\theta$ for any $\|x_1-x_2\|\leq \epsilon$ and some constants $L,\theta>0$.
If  $\nu_c,\rho_c$ satisfy
\begin{align}\label{feedbacknuc}
\begin{cases}
\nu_c&\leq \min\{\bar\nu_1,\bar \nu_2\},\\
\rho_c&\leq \min\{\bar \rho_2-\frac{\nu_c\bar\nu_1}{\bar\nu_1-\nu_c}, \bar \rho_1-\frac{\nu_c\bar\nu_2}{\bar\nu_2-\nu_c}\},
\end{cases}
\end{align}
where
\begin{align}\label{feedbackindi1}
\begin{cases}
\bar\nu_1&=\nu_1-\gamma_1 \tau(1+\lambda_1|\rho_1|)-\gamma_1^2\tau^2|\rho_1|,\\
\bar\rho_1&=\rho_1-\gamma_1\tau|\rho_1|/\lambda_1,
\end{cases}
\end{align}
\begin{align}\label{feedbackindi3}
\begin{cases}
\bar\nu_2&=\hat \nu_2(1-\frac{1}{\ell_2}),\\
\bar\rho_2&=\hat \rho_2-\frac{1}{4\ell_1}, \\
\bar\a_2&=\hat \a+ \tau\frac{m\mu^2}{4}(\ell_1+\hat\nu_2(\ell_2+1)),
\end{cases}
\end{align}
with $\lambda_1,\ell_1,\ell_2$ arbitrarily positive numbers and $\hat \nu_2,\hat \rho_2,\hat\a $ obtained from \eqref{indextau2}  by {\color{black}substituting} $\nu,\rho,\gamma$ with $\nu_2,\rho_2,\gamma_2$, respectively, then the {\color{black}closed-loop system} $\S_{c}$ satisfies the following passivity inequality
\begin{align}
&\frac{1}{\tau}(V(x[k+1])-V(x[k]))\nonumber\\
&\quad\quad \leq w^Ty-\nu_c w^Tw-\rho_c y^Ty+\bar \a_2/\tau,\label{feedinequality}
\end{align}
where $V(x[k])=V_1(x_1[k])+V_2(\hat x_2[k])$ and $x=\left[\begin{array}{c}x_1\\ \hat x_2\end{array}\right]$,  $w=\left[\begin{array}{c}w_1\\ w_2
\end{array}\right]$,  $y=\left[\begin{array}{c}y_1\\ \hat y_2
\end{array}\right]$ are the state, input, output of $\S_c$, respectively.
\end{thm}
\begin{proof} Given in Appendix \ref{sec:prf3}.
\end{proof}

Theorem \ref{thm:feedback} relates the passivity indices of the continuous time plant $\S_1$ and controller $\S_2$ to the passivity indices of discrete time system $\S_c$ where the symbolic controller $T_{\tau\mu\eta}(\S_2)$  is approximately bisimilar to $\S_2$. It can also be used as a guide to the design of the continuous controller, if the designer knows that only a software implementation will be used. Notice that the only requirement on the continuous controller is that it satisfies a certain passivity inequality.

A direct result of Theorem \ref{thm:feedback} and Theorem \ref{thm:bound} is that, if $\nu_2,\rho_2$ can be designed for $\Sigma_2$ such that \eqref{feedbacknuc} is satisfied with $\nu_c,\rho_c>0$,
and furthermore, if $\S_c$ is strongly finite-time detectable, then the states of $\S_c$ is ultimately bounded. 

Particularly, when the external input is zero (i.e., $w_1=w_2=0$), then \eqref{eqn:feed0} reduces to
\begin{align*}
&\quad \frac{1}{\tau}(V(x[k+1])-V(x[k]))\nonumber\\
&\leq  -(\bar \nu_2+\bar \rho_1)y_1^Ty_1-(\bar \nu_1+\bar \rho_2)\hat y_2^T\hat y_2+\bar\a_2\nonumber\\
& \leq -\min{\{\bar \nu_2+\bar \rho_1,\bar \nu_1+\bar \rho_2\}}y^Ty+\bar \a_2.
\end{align*}

If $\nu_2,\rho_2$ can be designed for $\Sigma_2$ such that $\bar \nu_2+\bar \rho_1,\bar \nu_1+\bar \rho_2>0$ and $\S_c$ is strongly finite-time detectable, then  $\S_c$ is ultimately bounded by Corollary \ref{cor:bound}.

{\color{black} \begin{rem} The symbolic model $T_{\tau\m\eta}(\S_2)$ we consider has countably infinite states. However, it is possible to consider compact subsets of state, input and output spaces and use the quantized version of these subsets as the state, input and output space of  $T_{\tau\m\eta}(\S_2)$. This will lead to a finite state model and local versions of the results presented above still hold. Similarly, it is straightforward to deal with local notions of passivity. \end{rem}}

\section{Example}\label{sec:example}

In this section, we illustrate the theoretical results in preceding sections by a simple cruise control example. 

The longitudinal dynamics of the vehicle is given by
\begin{align}
&\S_1:\quad \dot{x}_1=u_1-{\color{black}c_0x_1},\;y_1=x_1,\label{S1}
\end{align}
where $x_1$ is the speed of the vehicle, {\color{black}$c_0=0.01$ is the air-drag term}, $u_1$ is the scaled input and $y_1$ is the output.

The controller is given by
\begin{align}\label{o2}
\S_2:
\begin{cases}\dot{x}_2=\left[\begin{array}{cc}-1&-1\\1&-2
\end{array}\right]x_2+\left[\begin{array}{c}0\\ 1
\end{array}\right]u_2,\\
y_2=-\left[\begin{array}{cc}0.5& 0.5
\end{array}\right]x_2+2u_2.
\end{cases}
\end{align}


The passivity indices $\nu_1,\rho_1$ for $\S_1$ can be chosen as $\nu_1=0$, $\rho_1=0.01$ while  $\nu_2,\rho_2$ for $\S_2$ can be chosen as $\nu_2=0.31$, $\rho_2=0.42$. Let
$u_1=-y_2$, $u_2=y_1$. {\color{black}Then it is easy to find that $\nu_1+\rho_2>0,\rho_1+\nu_2>0$ and the feedback loop system is asymptotically stable.}

Note that the system \eqref{o2} is $\delta$-ISS. Therefore we can compute its $(\epsilon,\mu)$-approximate bisimilar system $T_{\tau\mu\eta}(\S_2)$ by choosing $\epsilon=0.9,\tau=0.3,\eta=0.1,\mu=0.1$.
Replace $\S_2$ with $T_{\tau\mu\eta}(\S_2)$ to constitute $\Sigma_c$
(as shown in Figure \ref{figSig2}), and assume the external input to be zero. Then by the formulas given in previous sections we can choose $\lambda_1=\lambda_2=\lambda_3=1,\lambda_4=1.5,\lambda_5=0.2,\ell_1=10,\ell_2=10$ such that $\bar\nu_1=-0.3039,\bar\rho_1=0.007$ and $\bar\nu_2=0.0106, \hat\rho_2=0.3411,\hat \a=0.6141$, which implies that $\bar\nu_1+\bar\rho_2>0$, $\bar\rho_1+\bar\nu_2>0$.
Simulation results in Figure \ref{figsimu} shows that $\S_c$ is ultimately bounded.
\begin{figure}[!hbt]
\begin{minipage}{.24\textwidth}
\includegraphics[width=4.4cm]{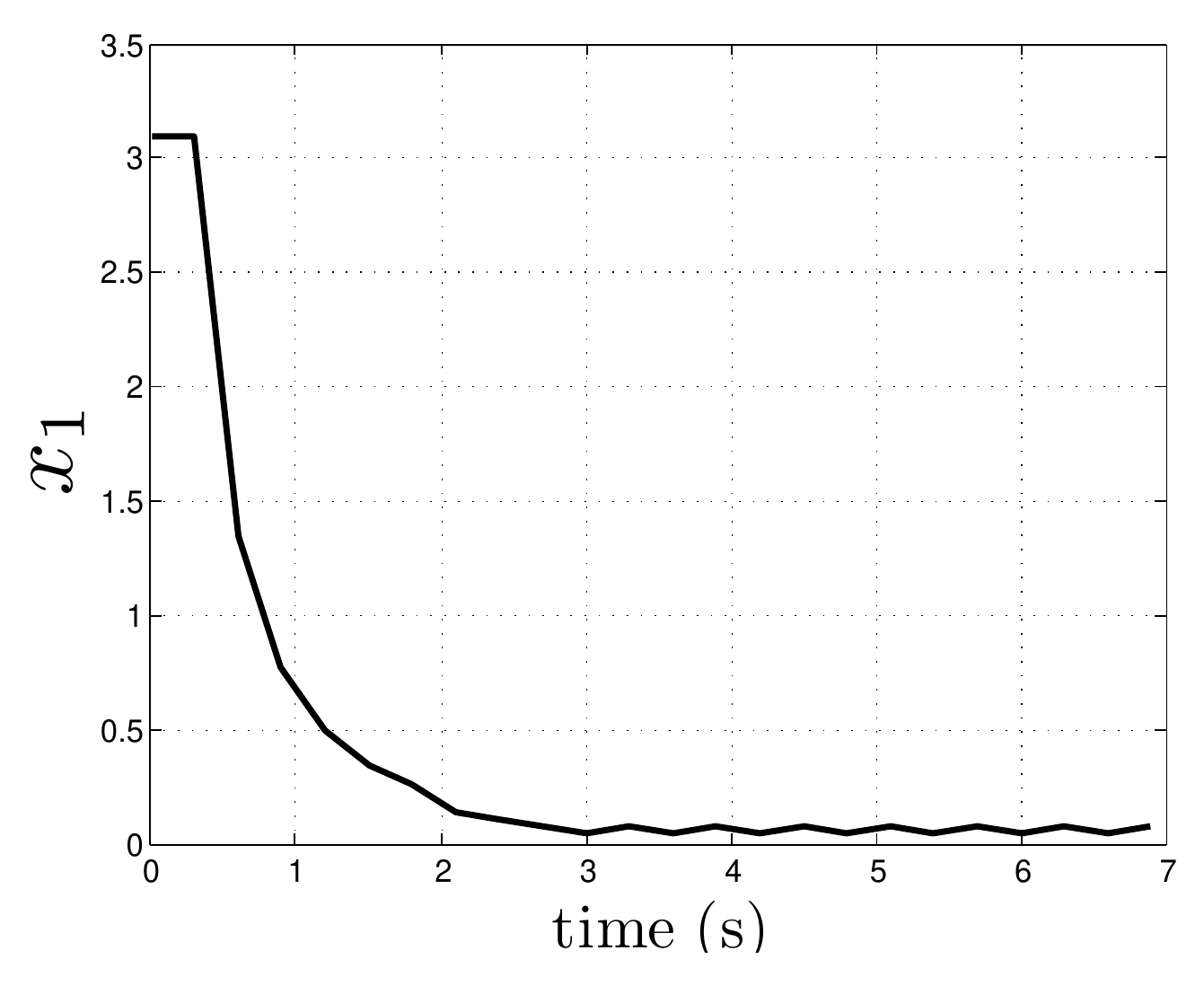}
\includegraphics[width=4.4cm]{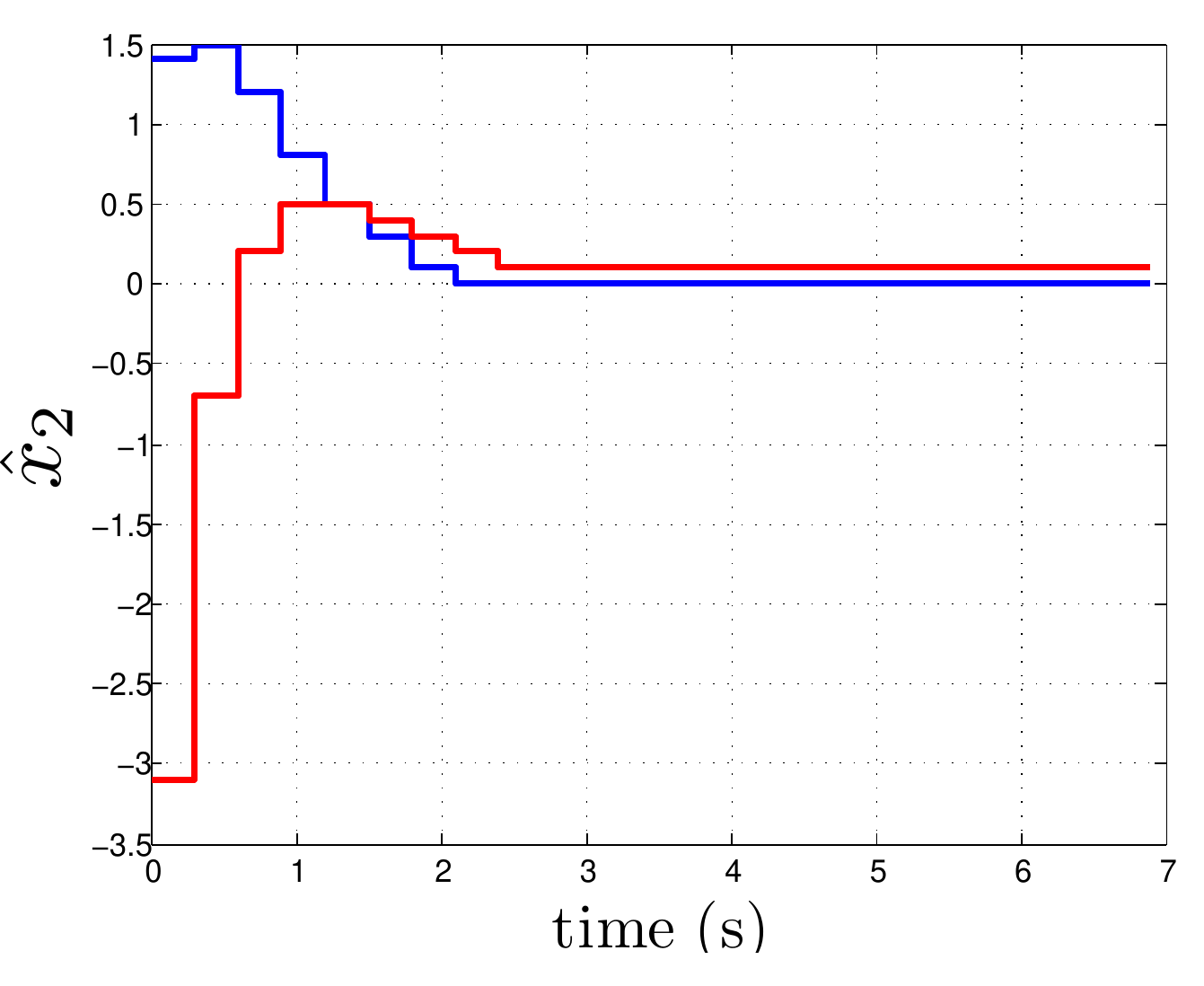}
    \end{minipage}
    \begin{minipage}{0.24\textwidth}
\includegraphics[width=4.4cm]{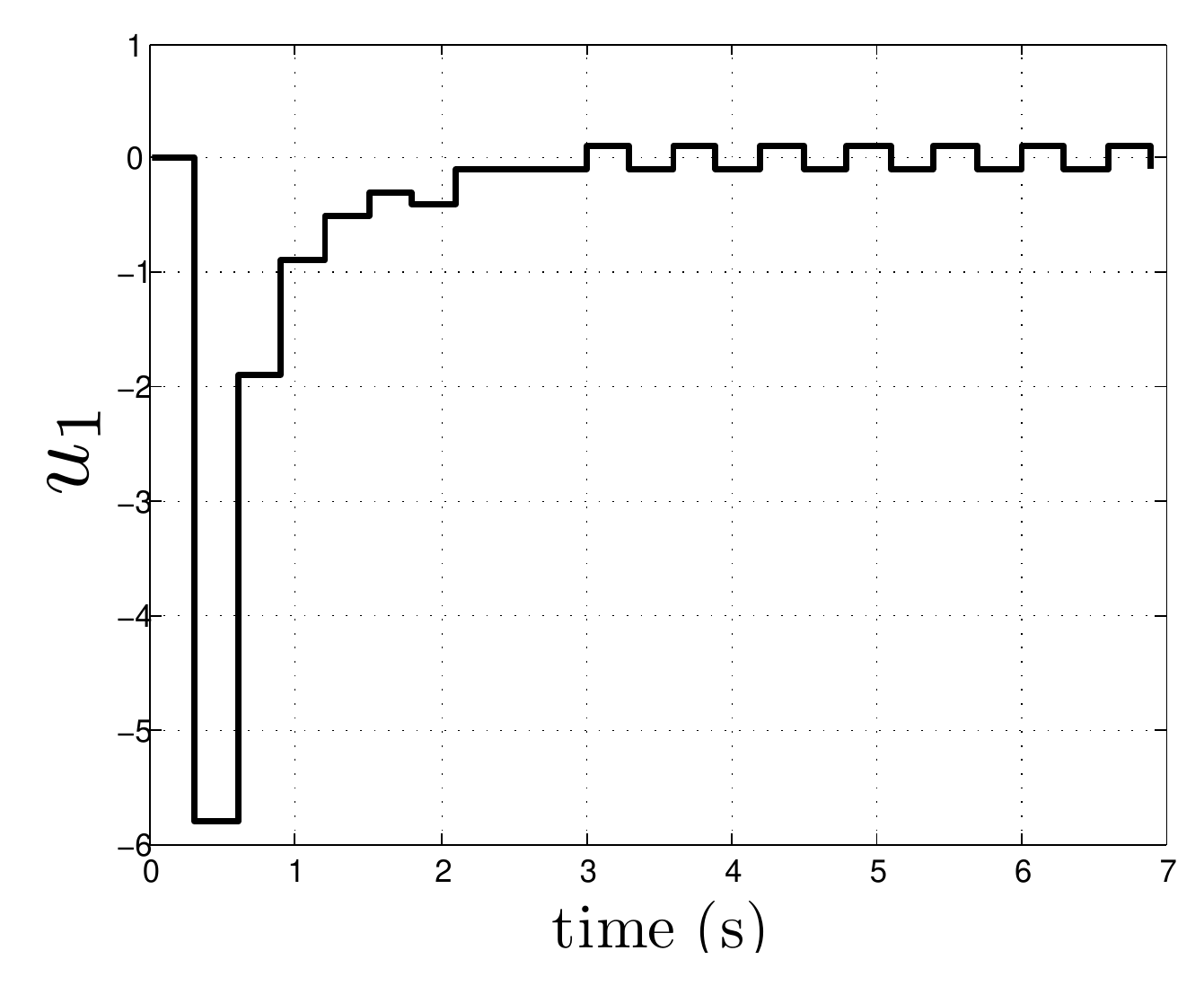}
\includegraphics[width=4.4cm]{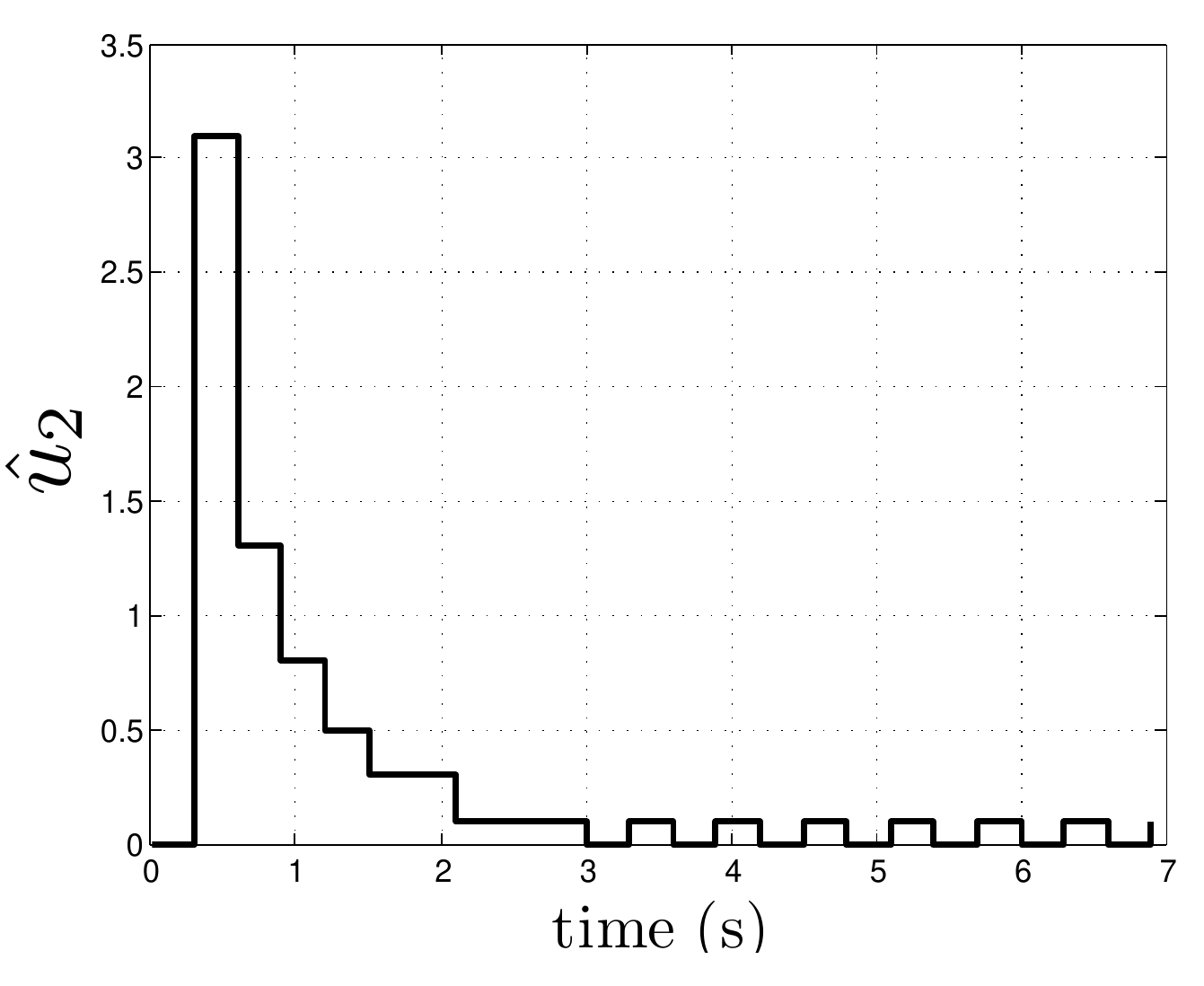}
    \end{minipage}
    \caption{Simulation result with initial state $x_1(0)=3.1,u_1(0)=0,\hat x_2(0)=(1.4,-3)',\hat u_2(0)=0$. (top left) $x_1(t)$; (top right) $u_1(t)$; (down left) blue is $\hat x_2^1[k]$  and red is $\hat x_2^2[k]$ ; (down right) $\hat u_2[k]$.}\label{figsimu}
    \end{figure}

\section{Conclusions}\label{sec:conclusions}

 We considered the problem of analyzing the passivity of a closed loop system when the controller is designed in continuous space, but implemented as a symbolic model. Our main result shows that if the implemented symbolic controller is obtained via approximate bisimulation, several passivity properties carry over despite such replacement. More precisely, we relate the passivity indices of the original system with both the plant and the controller as continuous systems, the bisimulation parameters, and the quasi-passivity indices of the new system with the controller implemented in discrete space. 

Combining ideas from symbolic control and passivity provides a general framework for analyzing and designing heterogeneous systems.
We are currently working on extensions of this framework to more general dissipation inequalities. A potentially limiting assumption in the paper is that every discrete time signal shares the same clock. Such synchronization runs counter to the promise of compositionality that passivity brings in large scale systems. We would like to relax this assumption in our future work. 

%


\appendices


\section{Proof of Theorem \ref{thm:bound}}\label{sec:prf1}

Let $\l=\frac{1}{2\rho}+2\nu$. Because
$$
u^T[k]y[k]\leq \frac{\l}{2}u^T[k]u[k]+\frac{1}{2\l}y^T[k]y[k],
$$
from equation \eqref{ieq:VSQP} we have
\begin{align}
&V(x[k+1])-V(x[k])\nonumber\\
\leq &\frac{\l}{2}u^T[k]u[k]+\frac{1}{2\l}y^T[k]y[k]-\nu u^T[k]u[k]-\rho y^T[k]y[k]+\a\nonumber\\
=&  (\frac{\l}{2}-\nu)u^T[k]u[k]+(\frac{1}{2\l}-\rho)y^T[k]y[k]+ \a\nonumber\\
\leq& \frac{1}{4\rho}B_1^2+\a.\label{quasi}
\end{align}

Therefore, for any $N_0\in\Z_+,k\in\Z_0$,
\begin{align}
& V(x[k+N_0])-V(x[k])\nonumber\\
\leq&  \sum_{i=k}^{k+N_0-1} ((\frac{\l}{2}-\nu)u^T[i]u[i]+(\frac{1}{2\l}-\rho)y^T[i]y[i]+\a)\nonumber\\
\leq & \sum_{i=k}^{k+N_0-1}
((\frac{\l}{2}-\nu)\|u[i]\|_2^2+\a) -(\rho-\frac{1}{2\l})\kappa\|x[k]\|_2^2\nonumber\\
\leq& N_0(\frac{B_1^2}{4\rho}+\a) -\frac{4\rho^2\nu\kappa}{4\rho\nu+1}\|x[k]\|_2^2\label{quasidetec}
\end{align}
where the second inequality is from \eqref{strongdetec}.

Define
\begin{align*}
r&:=\max\{B_2,\sqrt{\frac{N_0(\frac{B_1^2}{4\rho}+\a)(4\rho\nu+1) }{4\rho^2\nu\kappa}}\},\\
v&:=\max_{\|z\|_2\leq r}V(z),\\
D&:=\sup\{\|x\|_2:V(x)\leq v+N_0(\frac{B_1^2}{4\rho}+\a) \}.
\end{align*}

Clearly, $B_2\leq r\leq D$ and $\|x[k]\|_2\leq B_2\leq D <\infty$, {\color{black} where boundedness of $D$ follows from the fact that $V$ is radially unbounded}.

For $s\in\{k+1,...,k+N_0\}$, \eqref{quasi} implies that
\begin{align*}
V(x[s])&\leq V(x[k])+(s-k)(\frac{B_1^2}{4\rho}+\a)\\
&\leq v+N_0(\frac{B_1^2}{4\rho}+\a).\\
\end{align*}

Therefore, $\|x[s]\|_2\leq D$  for $s\in\{k+1,...,k+N_0\}$.

Now consider $s\in\{k+N_0+1,...,k+2N_0\}$. If there exists $s^*\in\{k+N_0+1,...,k+2N_0\}$ such that
\begin{equation}\label{bound}
V(x[s^*])>v+N_0(\frac{B_1^2}{4\rho}+\a),
\end{equation}
then by (\ref{quasi}) we have
$$
V(x[s^*-N_0])\geq V(x[s^*])-N_0(\frac{B_1^2}{4\rho}+\a) >v.
$$

It implies that $\|x[s^*-N_0]\|_2> r$. By (\ref{quasidetec}) we have
\begin{align*}
V(x[s^*])&\leq  V(x[s^*-N_0])+N_0(\frac{B_1^2}{4\rho}+\a) -\frac{4\rho^2\nu\kappa}{4\rho\nu+1}r^2\\
&\leq  V(x[s^*-N_0])\\
&\leq v+N_0(\frac{B_1^2}{4\rho}+\a).
\end{align*}

This contradicts with (\ref{bound}). Therefore, $\|x[s]\|_2\leq D$  for $s\in\{k+N_0+1,...,k+2N_0\}$.
By induction we conclude that $\sup_{s\geq k}\|x[s]\|_2\leq D$.
\hfill$\Box$

\section{Proof of Theorem \ref{thmpradis}} \label{sec:prf2}
Taking $t_1=k\tau,t_2=(k+1)\tau$ for some $k\in\Z_0$ in \eqref{con:dissipa}, we have:
\begin{align}\label{syspassin5}
&V(x((k+1)\tau))-V(x(k\tau))\leq\nonumber\\
&\int_{k\tau}^{(k+1)\tau}u(t)^Ty(t)-\nu u(t)^Tu(t)-\rho y(t)^Ty(t)\,\ud t.
\end{align}

(i)
For any $t\in[k\tau,(k+1)\tau)$,
\begin{align}
\|y(t)-y(k\tau)\|_2
&\leq \int_{s=k\tau}^{(k+1)\tau}\|\dot{y}(s)\|_2\;\ud s\nonumber\\
&\leq \sqrt{\tau}\sqrt{\int_{s=k\tau}^{(k+1)\tau}\|\dot{y}(s)\|^2_2\;\ud s}\nonumber\\
&\leq\sqrt{\tau}\sqrt{\gamma^2\int_{s=k\tau}^{(k+1)\tau}\|u(s)\|^2_2\;\ud s}\nonumber\\
&\leq \tau\gamma\|u[k]\|_2,\label{app2dis}
\end{align}


Because
\begin{align*}
&\left|\int_{k\tau}^{(k+1)\tau}u(t)^Ty(t)\;\ud t-\tau u[k]^Ty[k]\right|\\
\leq &\int_{k\tau}^{(k+1)\tau}\|u[k]\|_2\|y(t)-y[k]\|_2\;\ud t\\
\leq &\|u[k]\|_2\int_{k\tau}^{(k+1)\tau}\|y(t)-y(k\tau)\|_2\;\ud t\\
\leq &\tau^2\gamma\|u[k]\|_2^2,
\end{align*}

we have
\begin{align}
&\int_{k\tau}^{(k+1)\tau}u(t)^Ty(t)\;\ud t\leq\tau u[k]^Ty[k]+\tau^2\gamma\|u[k]\|_2^2. \label{D11}
\end{align}

It is also clear that
\begin{equation}\label{D12}
\int_{k\tau}^{(k+1)\tau}u(t)^Tu(t)\;\ud t=\tau u[k]^T u[k].
\end{equation}

Furthermore, because
\begin{align*}
& \left|\int_{k\tau}^{(k+1)\tau}y(t)^Ty(t)\;\ud t-\tau y[k]^T y[k]\right|\\
&= \int_{k\tau}^{(k+1)\tau}\|y(t)-y[k]\|_2^2+2\|y[k]\|_2\|y(t)-y[k]\|_2\;\ud t\\
&\leq \tau^3\gamma^2\|u[k]\|_2^2+2\tau^2\gamma\|y[k]\|_2\|u[k]\|_2\\
&\leq \tau^3\gamma^2\|u[k]\|_2^2+\tau^2\gamma(\frac{\|y[k]\|_2^2}{\lambda_1}+\lambda_1\|u[k]\|_2^2)
\end{align*}
where $\lambda_1$ is an arbitrary positive number, we have:
\begin{align}
& -\rho \int_{k\tau}^{(k+1)\tau}y(t)^Ty(t)\;\ud t\leq \tau^3\gamma^2|\rho|\|u[k]\|_2^2+\nonumber\\
&\tau(\frac{\tau\gamma|\rho|}{\lambda_1}-\rho)\|y[k]\|_2^2
+\lambda_1\tau^2\gamma|\rho|\|u[k]\|_2^2.\label{D13}
\end{align}

Combing (\ref{D11}), (\ref{D12}), (\ref{D13}) and noting that $V(x[k])=V(x(k\tau))$, the conclusion follows.

(ii) Similar to \eqref{app2dis}, for any $t\in[k\tau,(k+1)\tau)$, we have
$$
\|y(t)-y(k\tau)\|_2 \leq \tau\gamma\|\hat u[k]\|_2.
$$

Furthermore, because
\begin{align*}
\|y(k\tau)-\hat y[k]\|_2&=\|y[k]-\hat y[k]\|_2\\
&\leq \sqrt{m}\|y[k]-\hat y[k]\|\\
&\leq \sqrt{m}\mu/2,
\end{align*}
we have
\begin{align*}
\|y(t)-\hat y[k]\|_2&\leq \|y(t)-y(k\tau)\|_2+\|y(k\tau)-\hat y[k]\|_2\\
&\leq\tau\gamma\|\hat u[k]\|_2+\sqrt{m}\mu/2.
\end{align*}



Similar to the proof in (i), we have
\begin{align}
&\left|\int_{k\tau}^{(k+1)\tau}u(t)^Ty(t)\;\ud t-\tau\hat u[k]^T\hat y[k]\right|\nonumber\\
\leq &\int_{k\tau}^{(k+1)\tau}\|\hat u[k]\|_2\|y(t)-\hat y[k]\|_2\;\ud t\nonumber\\
\leq& \|\hat u[k]\|_2\int_{k\tau}^{(k+1)\tau}\tau\gamma\|\hat u[k]\|_2+\frac{\sqrt{m}\mu}{2}\;\ud t\nonumber\\
\leq& \tau^2\gamma\|\hat u[k]\|_2^2+\tau\frac{\sqrt{m}\mu}{4}(\frac{1}{\lambda_2}+
\lambda_2\|\hat u[k]\|^2_2),\nonumber
\end{align}
where $\lambda_2$ is an arbitrary positive number. It implies that
\begin{align}
&\int_{k\tau}^{(k+1)\tau}u(t)^Ty(t)\;\ud t\leq\tau\hat u[k]^T\hat y[k]\nonumber\\
&\quad+\tau(\tau\gamma+\frac{\lambda_2\sqrt{m}\mu}{4})\|\hat u[k]\|_2^2+\tau\frac{\sqrt{m}\mu}{4\lambda_2}.\label{D1}
\end{align}
%

It is clear that
\begin{equation}\label{D2}
\int_{k\tau}^{(k+1)\tau}u(t)^Tu(t)\;\ud t=\tau\hat u[k]^T\hat u[k].
\end{equation}

Furthermore,
\begin{align*}
& \left|\int_{k\tau}^{(k+1)\tau}y(t)^Ty(t)\;\ud t-\tau\hat y[k]^T\hat y[k]\right|\\
&\leq \int_{k\tau}^{(k+1)\tau}\|y(t)-\hat y[k]\|_2^2+2\|\hat y[k]\|_2\|y(t)-\hat y[k]\|_2\;\ud t\\
&\leq \tau(\tau\gamma\|\hat u[k]\|_2+\frac{\sqrt{m}\mu}{2})^2+2\tau\|\hat y[k]\|_2(\tau\gamma\|\hat u[k]\|_2+\frac{\sqrt{m}\mu}{2})\\
&\leq \tau^3\gamma^2\|\hat u[k]\|_2^2+\tau^2\sqrt{m}\mu\gamma
(\frac{1}{4\lambda_3}+\lambda_3\|\hat u[k]\|_2^2)+\tau \frac{m\mu^2}{4}\\
&+\tau^2\gamma
(\frac{1}{\lambda_4}\|\hat y[k]\|_2^2+\lambda_4\|\hat u[k]\|_2^2)+
\tau\sqrt{m}\mu(\frac{1}{4\lambda_5}+\lambda_5\|\hat y[k]\|_2^2)\\
&\leq \tau(\tau^2\gamma^2+\tau\sqrt{m}\mu\gamma\lambda_3+
\tau\gamma\lambda_4)\|\hat u[k]\|_2^2\\
&+\tau(\frac{\tau\gamma}{\lambda_4}+\sqrt{m}\mu\lambda_5)\|\hat y[k]\|_2^2+\tau(\frac{\tau\sqrt{m}\mu\gamma}{4\lambda_3}
+\frac{\sqrt{m}\mu}{4\lambda_5}+\frac{m\mu^2}{4})
\end{align*}
where $\lambda_3,\lambda_4,\lambda_5$ are arbitrary positive numbers. Then
\begin{align}
&-\rho\int_{k\tau}^{(k+1)\tau}y(t)^Ty(t)\;\ud t
\leq \tau^2|\rho|(\tau\gamma^2+\sqrt{m}\mu\gamma\lambda_3\nonumber\\
&\quad+\gamma\lambda_4)\|\hat u[k]\|_2^2+\tau(\frac{|\rho|\tau\gamma}{\lambda_4}+|\rho|\sqrt{m}\mu\lambda_5-\rho)
\|\hat y[k]\|_2^2\nonumber\\
&\quad+\tau\mu|\rho|(\frac{\tau\gamma\sqrt{m}}{4\lambda_3}
+\frac{\sqrt{m}}{4\lambda_5}+\frac{m\mu}{4}).\label{D3}
\end{align}

Finally, because
\begin{align*}
&|V(\hat x[k+1])-V(x((k+1)\tau))|\\
\leq & L\|\hat x[k+1]-x((k+1)\tau)\|^\theta
\leq  L\epsilon^\theta,\\
&|V(\hat x[k])-V(x(k\tau))|\\
\leq & L\|\hat x[k]-x(k\tau)\|^\theta
\leq  L\epsilon^\theta,
\end{align*}
 we have
\begin{align}\label{V1}
& V(\hat x[k+1])-V(\hat x[k])\nonumber \\
&\leq V(x((k+1)\tau))-V(x(k\tau))+2L\epsilon^\theta.
\end{align}

Combing (\ref{D1}), (\ref{D2}), (\ref{D3}) and \eqref{V1}, the conclusion follows immediately. \hfill$\Box$

\section{Proof of Theorem \ref{thm:feedback}} \label{sec:prf3}
Note that
\begin{align}
& V(x[k+1])-V(x[k])=(V_{1}(x_{1}[k+1])-V_{1}(x_{1}[k]))\nonumber\\
&\quad \quad +(V_{2}(\hat{x}_{2}[k+1])-V_{2}(\hat{x}_{2}[k])).\nonumber
\end{align}
We bound the two terms on the right hand side   individually.

\textbf{Bounding $V_{1}(x_{1}[k+1])-V_{1}(x_{1}[k])$:}
Consider the upper dashed block with input $u_1[k]$ and output $y_1[k]$. Noting that $u_1[k]=w_1[k]-\hat y_2[k]$, we have the following inequality by (i) of Theorem \ref{thmpradis}:
\begin{align}
& \frac{1}{\tau}(V_1(x_1[k+1])-V_1(x_1[k]))\leq (w_1[k]-\hat y_2[k])^Ty_1[k]\nonumber\\
&\quad-\bar \nu_1\|w_1[k]-\hat y_2[k]\|_2^2-\bar\rho_1\|y_1[k]\|_2^2\label{inequV2}
\end{align}
where $\bar\nu_1,\bar\rho_1$ are given by \eqref{feedbackindi1}.


\textbf{Bounding $V_{2}(\hat{x}_{2}[k+1])-V_{2}(\hat{x}_{2}[k])$:}
$T_{\tau\mu\eta}(\S_2)$  satisfies the following inequality by (ii)  of Theorem \ref{thmpradis}:
\begin{align*}
& \frac{1}{\tau}(V_2(\hat x_2[k+1])-V_2(\hat x_2[k]))\leq\nonumber \\
&\quad \hat u_2[k]^T\hat y_2[k]-\hat \nu_2 \|\hat u_2[k]\|_2^2
-\hat \rho_2 \|\hat y_2[k]\|_2^2 +\hat \a/\tau,
\end{align*}
where $\hat \nu_2,\hat \rho_2,\hat \a$ is obtained from \eqref{indextau2} by {\color{black}substituting} $\nu,\rho,\gamma$ with $\nu_2,\rho_2,\gamma_2$, respectively.

 Now, consider the lower dashed block with input $u_2[k]$ and output $\hat y_2[k]$. Because $\|\hat u_2[k]-u_2[k]\|\leq \mu/2$ under the uniform quantizer, we have $\|\hat u_2[k]-u_2[k]\|_2\leq \sqrt{m}\mu/2$. Following the arguments presented in the proof of Theorem~\ref{thmpradis}, we have
 {\color{black}
\begin{align*}
|\hat u_2[k]^T\hat y_2[k]-u_2[k]^T\hat y_2[k]|
&\leq \frac{\ell_1 m\mu^2}{4}+\frac{\|\hat y_2[k]\|^2_2}{4\ell_1},
\end{align*}
and
\begin{align*}
&|\hat u_2[k]^T\hat u_2[k]-u_2[k]^Tu_2[k]|
\leq \frac{m\mu^2}{4}+
\frac{\|u_2[k]\|^2_2}{\ell_2}+\frac{\ell_2m\mu^2}{4},
\end{align*}}
where $\ell_1,\ell_2$ are arbitrarily positive numbers. Therefore,
\begin{align}
&\quad \frac{1}{\tau}(V_2(\hat x_2[k+1])-V_2(\hat x_2[k]))\nonumber \\
&\leq\hat u_2[k]^T\hat y_2[k]-\hat \nu_2 \|\hat u_2[k]\|_2^2
-\hat \rho_2 \|\hat y_2[k]\|_2^2 +\hat \a/\tau\nonumber\\
%
%
&\leq u_2[k]^T\hat y_2[k]-\bar \nu_2 \|u_2[k]\|_2^2
-\bar \rho_2 \|\hat y_2[k]\|_2^2 +\bar \a_2/\tau\label{inequV1}
\end{align}
where $\bar \nu_2, \bar \rho_2,\bar \a_2$ are given by \eqref{feedbackindi3}.


With these two bounds, we see that (dropping the argument $k$ for notational simplicity)
\begin{align}
&\quad \frac{1}{\tau}(V(x[k+1])-V(x[k]))\nonumber\\
&\leq w_1^Ty_1+w_2^T\hat y_2+2\bar \nu_1 w_1^T \hat y_2-2\bar \nu_2 w_2^Ty_1-\bar\nu_1 w_1^Tw_1\nonumber\\
& \quad - \bar \nu_2w_2^Tw_2-(\bar \nu_2+\bar \rho_1)y_1^Ty_1-(\bar \nu_1+\bar \rho_2)\hat y_2^T\hat y_2\label{eqn:feed0}\\
&\leq w^Ty-\left[w_1^T\; \hat y_2^T\right]\left[\begin{array}{cc}\bar \nu_1& -\bar \nu_1\\-\bar \nu_1& \bar \rho_2+\bar\nu_1
\end{array}\right]\left[\begin{array}{c}w_1\\ \hat y_2
\end{array}\right]\nonumber\\
&\quad-\left[w_2^T\; y_1^T\right]\left[\begin{array}{cc}\bar \nu_2& \bar \nu_2\\\bar \nu_2& \bar \rho_1+\bar\nu_2
\end{array}\right]\left[\begin{array}{c}w_2\\ \hat y_1
\end{array}\right]+\bar \a_2/\tau.\label{eqn:feed1}
\end{align}

The final step follows by noting that if $\nu_c,\rho_c$ are chosen such that \eqref{feedbacknuc} holds (see also \cite{ZhuFeedbackACC14}), then


\begin{align*}
\left[\begin{array}{cc}\bar \nu_1-\nu_c& -\bar \nu_1\\-\bar \nu_1& \bar \rho_2+\bar\nu_1-\rho_c
\end{array}\right]&\geq 0,\\
\left[\begin{array}{cc}\bar \nu_2-\nu_c& \bar \nu_2\\\bar \nu_2& \bar \rho_1+\bar\nu_2-\rho_c
\end{array}\right]&\geq 0.
\end{align*}
This implies that for any $w_1,w_2,y_1,\hat y_2$, we have
\begin{align*}
&\left[w_1^T\; \hat y_2^T\right]\left[\begin{array}{cc}\bar \nu_1-\nu_c& -\bar \nu_1\\-\bar \nu_1& \bar \rho_2+\bar\nu_1-\rho_c
\end{array}\right]\left[\begin{array}{c}w_1\\ \hat y_2
\end{array}\right]\\
&\quad +\left[w_2^T\; y_1^T\right]\left[\begin{array}{cc}\bar \nu_2-\nu_c& \bar \nu_2\\\bar \nu_2& \bar \rho_1+\bar\nu_2-\rho_c
\end{array}\right]\left[\begin{array}{c}w_2\\ y_1
\end{array}\right]\geq 0,
\end{align*}
which is equivalent to
\begin{align}
&\left[w_1^T\; \hat y_2^T\right]\left[\begin{array}{cc}\bar \nu_1& -\bar \nu_1\\-\bar \nu_1& \bar \rho_2+\bar\nu_1
\end{array}\right]\left[\begin{array}{c}w_1\\ \hat y_2
\end{array}\right]\nonumber\\
&\quad +\left[w_2^T\; y_1^T\right]\left[\begin{array}{cc}\bar \nu_2& -\bar \nu_2\\-\bar \nu_2& \bar \rho_1+\bar\nu_2
\end{array}\right]\left[\begin{array}{c}w_2\\ y_1
\end{array}\right]\nonumber\\
&\geq \nu_c w^Tw+\rho_c y^Ty. \label {eqn:feed2}
\end{align}
Plugging \eqref{eqn:feed2} into \eqref{eqn:feed1}, we have the inequality  \eqref{feedinequality}.
\hfill$\Box$

\bibliographystyle{IEEEtran}
\bibliography{./CDC_pass}

\end{document}